\documentclass[12pt,a4paper]{article}

\usepackage{amsmath}
\usepackage{amsfonts}
\usepackage{amssymb}
\usepackage{amsthm}
\usepackage{bbm}
\usepackage[margin=0.7in]{geometry} 
\usepackage{setspace}

\newtheorem{theorem}{Theorem}
\newtheorem{lemma}[theorem]{Lemma}
\newtheorem{definition}{Definition}
\newtheorem*{remark}{Remark}
\newcommand{\isep}{\!\mathrel{{.}{.}\!}\nobreak}
\newcommand{\B}{\mathfrak{B}}
\newcommand{\C}{\mathbb{C}}
\newcommand{\E}{\mathbb{E}}
\newcommand{\R}{\mathbb{R}}
\newcommand{\G}{\mathcal{G}}

\newcommand{\MP}{\mathrm{MP},y}

\newcommand{\Op}{\Omega_p}

\renewcommand{\l}{\ell}
\newcommand{\Als}{A_\l(s)}
\newcommand{\Almp}{A_{\l,\MP}}

\newcommand{\Tr}{\textbf{Tr}}
\newcommand{\Var}{\mathrm{Var}}
\newcommand{\plp}{\Pi_{\l,p}}
\newcommand{\ga}{\gamma}

\newcommand{\Ga}{\Gamma}
\newcommand{\Om}{\Omega}
\newcommand{\om}{\omega}
\newcommand{\ogs}{\om_\ga(s)}
\newcommand{\ogos}{\om_{\ga_1}(s)}
\newcommand{\ogts}{\om_{\ga_2}(s)}
\newcommand{\Sp}{\Sigma_p}
\newcommand{\plsp}{\plp/\Sp}
\newcommand{\plsps}{\plp^2/\Sp}

\newcommand{\Vg}{V_\ga}
\newcommand{\vg}{v_\ga}
\newcommand{\lga}{\l_\ga}
\newcommand{\ov}{\Om(\Vg)}

\newcommand{\Wg}{W_\ga}
\newcommand{\Wgo}{W_{\ga_1}}
\newcommand{\Wgt}{\overline{W}_{\ga_2}}
\newcommand{\Wgu}{W_{\ga_1,\ga_2}}
\newcommand{\Vgo}{V_{\ga_1}}
\newcommand{\Vgt}{V_{\ga_2}}
\newcommand{\Vgu}{V_{\ga_1 \cup \ga_2}}
\newcommand{\Vgi}{V_{\ga_1 \cap \ga_2}}
\newcommand{\vgo}{v_{\ga_1}}
\newcommand{\vgt}{v_{\ga_2}}
\newcommand{\vgu}{v_{\ga_1 \cup \ga_2}}
\newcommand{\vgi}{v_{\ga_1 \cap \ga_2}}

\newcommand{\ovu}{\Om(\Vgu)}

\renewcommand{\d}{\mathrm{d}}

\begin{document}

\title{Spectral distribution of random matrices from mutually unbiased bases}
\author{Chin Hei Chan\thanks{C. Chan is at the Dept. of Mathematics, Hong Kong University of Science and Technology, Clear Water Bay, Kowloon, Hong Kong (email: chchanam@connect.ust.hk).} and Maosheng Xiong\thanks{M. Xiong is at the Dept. of Mathematics, Hong Kong University of Science and Technology, Clear Water Bay, Kowloon, Hong Kong (email: mamsxiong@ust.hk).}
}

\date{}

\maketitle
\begin{abstract}

We consider the random matrix obtained by picking vectors randomly from a large collection of mutually unbiased bases of $\C^n$, and prove that the spectral distribution converges to the Marchenko-Pastur law. This shows that vectors in mutually unbiased bases behave like random vectors. This phenomenon is similar to that of binary linear codes of dual distance at least 5, which was studied in previous work.

\end{abstract}

\section{Introduction}


Random matrix theory is the study of matrices whose entries are random variables. Of particular interest is the study of eigenvalue statistics of random matrices such as the empirical spectral distribution. This has been broadly investigated in a wide variety of areas, including statistics \cite{Wis}, number theory \cite{MET}, economics \cite{econ}, theoretical physics \cite{Wig} and communication theory \cite{TUL}.

Most of the matrix models considered in the literature were matrices whose entries are independent random variables. In a series of papers (see \cite{Tarokh2,OQBT}), initiated in \cite{Tarokh1}, the authors studied the behaviour of sample-covariance type matrices formed by randomly selecting codewords from binary linear codes, and among many other things, they proved that such matrices behave like truly random matrices with respect to the empirical spectral distribution, as long as the minimum Hamming distance of the dual code is at least 5. More precisely, the limiting spectral distribution converges to the Marchenko-Pastur (MP) law. This result can be considered as a joint randomness test on sequences derived from binary linear codes. It is called a ``group randomness'' property of the code and may have potential applications, for example, in efficiently generating random matrices by using binary linear codes of dual distance at least 5. 

Since these work, some other aspects of group randomness properties of linear codes have been studied and interesting results have been obtained in recent years. For example, \cite{Tarokh3} studied the matrix which is the product of pseudorandom matrices arising from two different linear codes; \cite{Sol1,Sol2} studied the Wigner type matrix generated from from $m$-sequences and BCH codes of large dual distance respectively; \cite{CSC} studied matrices generated from linear codes as before but normalized differently. In all these cases, simple conditions were found in terms of some parameters of the linear codes that ensure that the matrices obtained behave like random matrices of independent entries with respect to the empirical spectral distribution.


In this paper we study the group randomness property of mutually unbiased bases (MUBs). A collection of orthonormal bases $B_1,\ldots, B_m$ of the vector space $\C^n$ ($m,n \ge 2$) is called mutually unbiased if for any $v_i \in B_i, v_j \in B_j$ where $i \ne j$, we have
\begin{eqnarray} \label{1:eq1}
\left|\langle v_i,v_j \rangle\right| = \frac{1}{\sqrt{n}}\, .
\end{eqnarray}
Here $\langle v_i,v_j \rangle$ is the standard Hermitian inner product on the complex vector space $\C^n$.

The notion of MUBs emerged in the literature of quantum mechanics from the work of Schwinger \cite{SCH} and has found important applications in quantum information theory, in particular in quantum state determination \cite{IVO} and in quantum cryptography \cite{BEC,BEN,BRU,SCH}. MUBs are also closely related to many other combinatorial objects such as spherical 2-design \cite{Gsl, KLA1}, semifields \cite{Gsl}, orthogonal Latin squares \cite{Hall,Woc} and planar functions \cite{Hall} etc.

Denote by $N(n)$ the maximal number of orthonormal bases of $\C^n$ which are MUBs. It is well-known $N(n) \le n+1$ (\cite{BAN,DEL,HOG,KAB,WOO}). The extremal set that achieves the equality is called complete MUBs and has important applications in quantum computation \cite{IVO}. So far it is known that $N(n)=n+1$ when $n$ is a prime power (see \cite{BAN,IVO,WOO}), and some explicit constructions can be seen in \cite{KLA}. However, very little is know about the exact value of $N(n)$ if $n$ is not a prime power. In particular, even the value $N(6)$ is not known. 

It is conceivable that vectors in MUBs of $\C^n$ are in general positions and spread quite uniformly on the unit sphere $\mathbb{S}^{2n-1}$, hence they should satisfy some group randomness properties. In this paper, we prove that this is indeed the case with respect to the empirical spectral distribution.

\subsection{Statement of the main result}
To state the main result, we need some notation.

For an $n \times n$ matrix $\mathbf{A}$, let $\lambda_1,\ldots, \lambda_n$ be its eigenvalues. The \emph{spectal measure} of $\mathbf{A}$ is given by
\[\mu_{\mathbf{A}}:=\frac{1}{n}\sum_{i=1}^n \delta_{\lambda_i},\]
where $\delta_z$ is the Dirac measure at $z$. The empirical spectral distribution of $\mathbf{A}$ is defined as
\[F_{\mathbf{A}}(x):=\int_{-\infty}^{x} \mu_{\mathbf{A}} \,(\d x). \]
The main result of this paper is as follows.
\begin{theorem}\label{thmMP}
Let $B_1,\ldots,B_m$ be \emph{MUBs} of $\C^n$ with $m \ge \sqrt{n}$. Denote $\B=\sqcup_{i=1}^m B_i$. Choosing $p$ vectors uniformly and independently from $\B$, we obtain a $p \times n$ random matrix $\Phi_n$. Denote by $F_{\G_n}$ the  empirical spectral distribution of the Gram matrix $\G_n:=\Phi_n\Phi_n^*$, where $\Phi_n^*$ is the conjugate transpose of $\Phi_n$. Then for any $x \in \R$, as $n \to \infty$ with $y=\frac{p}{n} \in (0,1)$ fixed, we have
$$ F_{\G_n}(x) \to F_{\MP}(x) \quad \textbf{in probability}.$$
Here $F_{\MP}$ is the cumulative distribution function of the Marchenko-Pastur law whose density function is given by
\begin{equation*}\label{mppdf}
f_{\MP}(x)=\frac{1}{2\pi xy}\sqrt{(b-x)(x-a)}\, \mathbbm{1}_{[a,b]}\, \d x\,,
\end{equation*}
and the constants $a$ and $b$ are defined as
\begin{equation*}\label{ab}
a=(1-\sqrt{y})^2, \quad b=(1+\sqrt{y})^2,
\end{equation*}
and $\mathbbm{1}_{[a,b]}$ is the indicator function of the interval $[a,b]$.
\end{theorem}

\begin{remark} 1). It is well-known that as the dimension grows to infinity, the empirical spectral distribution of the Gram matrix of real i.i.d. random matrices follows the Marchenko-Pastur law \cite{Ander, MET}, the above result can be interpreted as a joint randomness test for vectors in \emph{MUBs}. This is similar to \cite{Tarokh2, OQBT} where random matrices from binary linear codes were considered.


2). The notion of approximately mutually unbiased bases (or \emph{AMUBs}) was introduced by Shparlinski and Winterhof \cite{Shpar} by relaxing equality (\ref{1:eq1}) to
\begin{eqnarray} \label{1:eq2}
\left|\langle v_i,v_j \rangle\right| = O\left(\sqrt{\frac{\log n}{n}}\right)\, .
\end{eqnarray}
They showed that there are $n+1$ \emph{AMUBs} for any $n$ by using exponential sums. They also constructed $n+1$ \emph{AMUBs} where the norm of the inner product in (\ref{1:eq2}) is replaced by $O\left(\frac{1}{\sqrt{n}}\right)$ for almost all dimensions $n$, and their construction can be extended to all dimensions $n$ by assuming certain conjectures about the gap between consecutive primes. Some other variants of \emph{AMUBs} have been studied in \cite{Cao, LiJ, Wang}. It can be seen that Theorem \ref{thmMP} also holds true for \emph{AMUBs}. For the sake of simplicity, however, in the paper we only consider \emph{MUBs}.
\end{remark}

For the proof of Theorem \ref{thmMP}, we use the moment method: we compute all the $\l$-th moments and the variance of the spectral distribution (Theorems \ref{momentMP} and \ref{VarMP}) and compare them with the Marchenko-Pastur law. The main computation relies crucially on a technical counting result (Lemma \ref{4:wglem}), for which the desired estimate is obtained by a graph method. This is similar to \cite{OQBT} where a graph method was used to obtain a crucial estimate by exploiting the algebraic property of binary linear codes with dual distance at least 5. In this paper, the special structure of MUBs will play important roles in all the proofs.

This paper is organized as follows: in Section \ref{secMP} we outline the moment method and compute the $\l$-th moment of the special distribution (Theorem \ref{momentMP}), and in Section \ref{secVarMP} we compute the variance (Theorem \ref{VarMP}). This would finish the proof of Theorem \ref{thmMP} directly. A technical counting result (Lemma \ref{4:wglem}) was applied in both Section  \ref{momentMP} and Section \ref{secVarMP}. To streamline the presentation of this paper, we postpone the proof of this technical result to {Appendix}.


\section{The moment method}\label{secMP}
Recall that $\B=\sqcup_{i=1}^m B_i$ where $B_1,\ldots, B_m$ are MUBs of $\C^n$ where $m \geq \sqrt{n}$. Denote by $\Op$ the set of all maps $s: [1\isep p] \to \B$. $\Op$ is a probability space endowed the uniform probability, corresponding to selecting p vectors from $\B$ uniformly and independently. Here $[1\isep p]$ denotes the set of all integers between 1 and $p$. It is easy to see that $\#\Op=(mn)^p$.

For each $s \in \Op$, the corresponding $p \times n$ matrix $\Phi(s)$ is given by
\begin{equation*}\label{Phi}
\Phi(s)^*=\left[s(1)^*, s(2)^*, \ldots, s(p)^* \right].
\end{equation*}
Here each $s(i)$ is a row vector and $s(i)^*$ is the conjugate transpose of $s(i)$. Denote
\begin{equation*}\label{G}
\G(s)=\Phi(s)\Phi(s)^*.
\end{equation*}
This is a $p \times p$ Hermitian matrix whose $(i,j)$-th entry is given by $\langle s(i),s(j)\rangle$.

Let $\lambda_1(s),\lambda_2(s),\cdots,\lambda_p(s)$ be the eigenvalues of $\G(s)$. Given any positive integer $\l$, define
\begin{eqnarray} \label{2:al} \Als:=\frac{1}{p}\sum_{i=1}^p \lambda_i(s)^\l=\frac{1}{p}\Tr\left(\G(s\right)^\l).
\end{eqnarray}
This is the $\l$-th moment of the empirical spectral distribution of $\G(s)$. Here $\Tr\left(\G(s)^\l\right)$ is the trace of the matrix $\G(s)^\l$.

Denote by $\E(\cdot,\Op)$ and by $\Var(\cdot,\Op)$ the expectation and variance of a random variable in the probability space $\Op$ respectively. To prove Theorem \ref{thmMP}, it suffices to prove the following two statements (see \cite{RMT}):

(i) $\E(\Als,\Op) \to \Almp$ as $n \to \infty$, where $\Almp$ is the $\l$-th moment of the corresponding Marchenko-Pastur law $F_{\mathrm{MP},y}$, which is given by (see \cite{RMT})
$$\Almp=\sum_{i=0}^{\l-1}\frac{y^i}{i+1}\binom{\l}{i}\binom{\l}{i-1};$$

(ii) $\Var(\Als,\Op) \to 0$ as $n \to \infty$.

Actually, we prove the following:
\begin{theorem}\label{momentMP}
For any fixed positive integer $\l$, we have
\begin{eqnarray} \label{2:thm2eq}
\E\left(\Als,\Op\right)=\sum_{i=0}^{\l-1}\frac{y^i}{i+1}\binom{\l}{i}\binom{\l}{i-1}+O_\l\left(\frac{1}{m}+\frac{1}{n}\right).\end{eqnarray}
Here the constant implied in the big $O_{\l}$-symbol depends only on $\l$.
\end{theorem}
\begin{theorem}\label{VarMP}
For any fixed positive integer $\l$, we have
$$\Var(\Als,\Op)=O_\l\left(\frac{1}{mn}+\frac{1}{n^2}\right).$$
\end{theorem}
The rest of this section is devoted to a proof of Theorem \ref{momentMP}. We leave the proof of Theorem \ref{VarMP} to the next section.

\subsection{Problem Set-up}
A map $\ga: [0\isep \l] \to [1\isep p]$ is called a closed path if $\ga(0)=\ga(\l)$. Denote by $\plp$ the set of all closed paths from $[0\isep \l]$ to $[1\isep p]$.

Now for any $s \in \Op$ and $\ga \in \plp$, we define
\begin{equation}\label{ogs}
\ogs:= \prod_{i=0}^{l-1} \left\langle s\circ\ga(i),s\circ\ga(i+1) \right\rangle .
\end{equation}
We can expand $\Tr\left(\G(s)^\l \right)$ on the right side of (\ref{2:al}) as
$$\Als=\frac{1}{p}\sum_{\ga \in \plp} \ogs.$$
This implies
$$\E(\Als,\Op)=\frac{1}{p}\sum_{\ga \in \plp} \E(\ogs,\Op).$$
To simplify the above equation a little further, we use an argument from \cite{OQBT}.

Let $\Sp$ be the group of permutations on the set $[1\isep p]$. Then $\Sp$ acts on $\plp$, since $\sigma \circ \gamma \in \plp$ whenever $\gamma \in \plp$ and $\sigma \in \Sp$. Let $[\gamma]$ be the equivalent class of $\gamma$, that is,
\[[\gamma]=\left\{\sigma \circ \gamma: \sigma \in \Sp \right\}. \]
We may write
\[\E(\Als,\Op)=\frac{1}{p}\sum_{\gamma \in \plp/\Sp} \,\,\, \sum_{\tau \in [\gamma]}\E\left(\omega_{\tau}(s), \Op\right), \]
where $\plsp$ is the set of representatives of equivalence classes under the equivalence relation
$$\gamma \sim \gamma' \iff \gamma=\sigma \circ \gamma' \quad \exists \sigma \in \Sp \, .$$
For any $\gamma \in \plp /\Sp$, one can easily see that
\[\E\left(\omega_{\tau}(s), \Omega_p\right)=\E\left(\omega_{\gamma}(s), \Omega_p\right), \quad \forall \tau \in [\gamma].  \]
Moreover, let
$$\Vg=\ga\left([0\isep \l] \right) \subset [1\isep p], \quad \vg=\#\Vg, $$
and define the probability space
\begin{eqnarray} \label{2:ov} \ov:=\left\{s: \Vg \to \B \right\}\end{eqnarray}
endowed with the uniform probability. It is clear that $\#[\gamma]=\frac{p!}{(p-v_{\gamma})!}$, $\#\Omega(V_{\gamma})=(mn)^{\vg}$ and
\[\E\left(\omega_{\gamma}(s), \Omega_p\right)=\E\left(\omega_{\gamma}(s), \Omega(V_{\gamma})\right). \]
Summarizing the above considerations, we have
\begin{equation*}\label{Emom}
\E(\Als,\Op)=\frac{1}{p}\sum_{\ga \in \plsp} \frac{p!}{(p-\vg)!} \,\, \Wg,
\end{equation*}
where for simplicity we define $\Wg$ by
\begin{equation}\label{Wg}
\Wg:=\E\left(\ogs,\ov\right).
\end{equation}


\subsection{Proof of Theorem \ref{momentMP}}

The evaluation of $\Wg$ as defined in (\ref{Wg}) is technical and involves a combinatorial argument. To streamline the proof of Theorem \ref{momentMP}, we postpone the study of $\Wg$ to {\bf Appendix} (see Lemma \ref{4:wglem}). Here we assume Lemma \ref{4:wglem} instead and prove Theorem \ref{momentMP}.

Recall from Lemma \ref{4:wglem} in {\bf Appendix} that there is a subset $\Ga_l \subset \plsp$ such that
\begin{eqnarray} \label{2:wgamma} \Wg=\begin{cases}
n^{1-\vg} &\mbox{ if } \ga \in \Ga_\l\,  \\
O_{\l}\left(n^{1-\vg}\left(m^{-1}+n^{-1}\right)\right) & \mbox{ if } \ga \notin \Ga_\l\, .
\end{cases}\end{eqnarray}
Using (\ref{2:wgamma}) and the fact that
\begin{equation}\label{sumpath}
\sum_{\substack{\ga \in \plsp \\ \vg=v}} 1 < v^\l \leq \l^\l, \quad \forall \, v \le \l,
\end{equation}
we obtain
\begin{align} \label{2:est}
\E(\Als,\Op)&=\frac{1}{p}\sum_{\ga \in \Ga_\l} \frac{p!}{(p-\vg)!}n^{1-\vg}+E_\l,
\end{align}
where
\begin{align}
|E_\l | & = \frac{1}{p}\sum_{\ga \in \plp /\Sp} \frac{p!}{(p-\vg)!} O_\l\left(n^{1-\vg}\left(\frac{1}{m}+\frac{1}{n}\right)\right) \nonumber \\
&=O_\l \left( \frac{1}{m}+\frac{1}{n} \right). \label{2:elest}
\end{align}
As for the main term, using the identity (see \cite[Section IV-E]{OQBT} and \cite[Lemma 3.4]{RMT})
$$\sum_{\substack{\ga \in \Ga_\l\\ \vg=v}} 1=\frac{1}{v}\binom{\l}{v-1}\binom{\l-1}{v-1},$$
we can easily obtain
\begin{align}
\frac{1}{p}\sum_{\ga \in \Ga_\l} \frac{p!}{(p-\vg)!}n^{1-\vg} &=\sum_{\ga \in \Ga_\l} \left(\frac{p}{n}\right)^{\vg-1}\left(1+O_\l\left(\frac{1}{p}\right)\right)  \nonumber \\
&=\sum_{v=1}^\l \frac{y^{v-1}}{v}\binom{\l}{v-1}\binom{\l-1}{v-1} +O_\l\left(\frac{1}{n}\right). \label{2:emain}
\end{align}
Inserting (\ref{2:emain}) and (\ref{2:elest}) into Equation (\ref{2:est}) yields the desired result (\ref{2:thm2eq}). The completes the proof of Theorem \ref{momentMP}.

\section{Study of the variance} \label{secVarMP}

Now we proceed to prove Theorem \ref{VarMP}. We first expand the quantity $\Var(\Als,\Op)$ as
\begin{align}
\Var(\Als,\Op)&=\E\left(\left|\Als \right|^2,\Op\right)-\left|\E\left(\Als,\Op\right)\right|^2\nonumber\\
&=\frac{1}{p^2}\sum_{\ga_1,\ga_2 \in \plp} \left(\E(\ogos\overline{\ogts},\Op)-\E(\ogos,\Op)\,\overline{\E(\ogts,\Op)}\right). \label{3:Var}
\end{align}
Here $\overline{z}$ is the complex conjugate of $z$. To simplify it further, denote by $\plsps$ the set of representatives of equivalence classes of the pairs $(\ga_1,\ga_2) \in \plp^2$ under the equivalence relation
$$(\ga_{11},\ga_{12}) \sim (\ga_{21},\ga_{22}) \iff (\ga_{11},\ga_{12})=\left(\sigma \circ \ga_{21},\sigma \circ \ga_{22} \right) \quad \exists \sigma \in \Sp\, ,$$
and for any $\gamma_1, \gamma_2 \in \plp$, define
$$\Vgu:=\Vgo \cup \Vgt, \quad \Vgi:=\Vgo \cap \Vgt, \quad\vgu:=\#\Vgu,\quad\vgi=\#\Vgi\, .$$
Using similar arguments as before, we can write Equation (\ref{3:Var}) as
\begin{align}
\Var(\Als,\Op)&=\frac{1}{p^2}\sum_{(\ga_1,\ga_2) \in \plsps}\frac{p!}{(p-\vgu)!} \left(\Wgu-\Wgo\Wgt \right). \label{Var}
\end{align}
Here
\begin{equation}\label{Wgu}
\Wgu:=\E\left(\ogos\overline{\ogts},\ovu\right),
\end{equation}
and $\Wg$ is defined in (\ref{Wg}).

\begin{lemma} \label{3:lem2}
For any $(\ga_1,\ga_2) \in \plsps$, we have
\begin{equation}\label{diff}
\Wgu-\Wgo\Wgt \ll_{\l} n^{1-\vgu}\left(\frac{1}{m}+\frac{1}{n}\right).
\end{equation}
\end{lemma}
\begin{proof}
If $\vgi=0$ or equivalently $\Vgi=\emptyset$, then clearly $\Wgu=\Wgo\Wgt$, so (\ref{diff}) holds.

Now we consider the case that $\vgi \geq 1$. By choosing different starting points if necessary, we may assume that $\ga_1(0)=\ga_2(0)$. Joining $\ga_1$ and $\ga_2$, we define a new path $\ga':[0\isep 2\l] \to [1\isep p]$ by setting
\begin{equation}\label{ga12}
\ga'(i)=\begin{cases}
\ga_1(i) &\mbox{ if } 0 \leq i \leq \l, \\
\ga_2(2\l-i) &\mbox{ if } \l \leq i \leq 2\l.
\end{cases}
\end{equation}
It is easy to see that $\ga'$ is a closed path of length $2\l$, with the number of vertices given by
$$\#\ga'\left([0 \isep 2 \l]\right)=\vgu=v_{\ga_1}+v_{\ga_2}-\vgi.$$
Moreover, we have $\Wgu=W_{\ga'}$, whose value can be obtained directly from (\ref{2:wgamma}), depending on whether or not $\ga' \in \Gamma_{2\l}$.


Now suppose $\ga' \in \Ga_{2\l}$. by the structure of $\Ga_{2\l}$ (see {\bf Appendix} or \cite[Section IV]{OQBT}), $\ga'$ corresponds to a double-tree, that is, the skeleton of $\ga'$ is a tree (\emph{there is no cycle in the skeleton of the graph $\ga'(0) \to \ga'(1) \to \ldots \to \ga'(2\l-1) \to \ga'(0)$}), and each edge is traversed exactly twice. Since $\ga'$ is obtained by joining $\ga_1$ with $\ga_2$, it is easy to see that $\ga_1,\ga_2 \in \Ga_\l$. Moreover, we must have $\vgi=1$: if $\vgi \ge 2$ instead, then by considering the path between two overlapping vertices in $\Vgi$, we see that it either forms a cycle or the edges involved are traversed at least four times in $\ga'$, contradicting the condition that $\ga' \in \Ga_{2\l}$ (so each edge is traversed exactly twice). Thus $\vgu=v_{\ga_1}+v_{\ga_2}-1$, and by using (\ref{2:wgamma}) we have
\[\Wgu-\Wgo\Wgt=n^{1-\vgu}-n^{1-v_{\ga_1}}n^{1-v_{\ga_2}}=0. \]

Next we assume that $\ga' \notin \Ga_{2\l}$ and $\vgi \geq 1$. Then either $\ga_1 \notin \Ga_\l$ or $\ga_2 \notin \Ga_\l$. By using (\ref{2:wgamma}) again we have
$$\Wgu=W_{\ga'} \ll_{\l}n^{1-\vgu}\left(\frac{1}{m}+\frac{1}{n}\right), $$
and
$$\Wgo\Wgt \ll_{\l}n^{1-\vgo}n^{1-\vgt}\left(\frac{1}{m}+\frac{1}{n}\right)
\le n^{1-\vgu}\left(\frac{1}{m}+\frac{1}{n}\right), $$
so we still have (\ref{diff}). 


Summarizing all the above cases, we conclude that Equation (\ref{diff}) holds true for any $(\ga_1,\ga_2) \in \plsps$. This completes the proof of Lemma \ref{3:lem2}.
\end{proof}

Finally, inserting (\ref{diff}) into (\ref{Var}) and using 
\begin{equation*}\label{bound2}
\sum_{\substack{(\ga_1,\ga_2) \in \plsps \\ \vgu=v}} 1 < v^{2\l} \leq (2\l)^{2\l},
\end{equation*}
we can obtain
\begin{align*}
\Var(\Als,\Op)&\ll_{\l}\frac{1}{p^2}\sum_{(\ga_1,\ga_2) \in \plsps}p^{\vgu}n^{1-\vgu}\left(\frac{1}{m}+\frac{1}{n}\right)\\
&\ll_{\l} \sum_{v=1}^{2\l}y^{v-2}\left(\frac{1}{mn}+\frac{1}{n^2}\right)\sum_{\substack{(\ga_1,\ga_2) \in \plsps\\ \vgu=v}} 1\\
&\ll_{\l}\frac{1}{mn}+\frac{1}{n^2}.
\end{align*}
This completes the proof of Theorem \ref{VarMP}.

\section{Appendix}

In this Section we prove (\ref{2:wgamma}) (see also Lemma \ref{4:wglem}), which plays important roles in the proofs of Theorem \ref{momentMP} and \ref{VarMP} in the previous sections.

Let $\gamma: [0 \isep l_{\gamma}] \to [1 \isep p]$ be a closed path with 
$$V_{\gamma}=\gamma([0 \isep l_{\gamma}])=\{z_a: 1 \le a \le v_{\gamma}\}, \quad v_{\gamma}=\#V_{\gamma}.$$ 
Denote $I_a=\gamma^{-1}(z_a)$ for any $1 \le a \le v_{\gamma}$. Recall that $\Wg=\E\left(\ogs,\ov\right)$ where $\ogs$ is defined in (\ref{ogs}), and $\ov$ is defined in (\ref{2:ov}).

\begin{definition} \label{4:def1}
The closed path $\gamma$ is called reduced if $l_{\gamma}=v_{\gamma}=1$, or if $v_{\gamma} \ge 2$ and the following two conditions are satisfied:
\begin{itemize}
\item[(i).] each $\#I_a  \ge 2$, hence $l =\sum_a \# I_a \ge 2 v \ge 4$;

\item[(ii).] each $I_a$ does not contain consecutive indices, that is, $\gamma(u) \ne \gamma(u+1) \,\, \forall u$.

\end{itemize}
\end{definition}

\subsection{$\ga$ is reduced}

Suppose $\ga$ is reduced. If $\lga=\vg=1$, then obviously $\ogs=1$, so we have $\Wg=1$.

Now we consider the case that $\ga$ is reduced and $\vg \geq 2$. For each given $s \in \ov$, define
\[N_{s,\ga}:=\#\left\{s(z): z \in \Vg\right\}, \quad C_{s,\ga}:=\#\left\{i: s \circ \ga(i) \ne s \circ \ga(i+1)\right\}. \]
Since for any $v_1,v_2 \in \B$,
$$\langle v_1, v_2\rangle=\begin{cases}
1 &\mbox{ if } v_1=v_2, \\
O(n^{-\frac{1}{2}}) &\mbox{ if } v_1 \neq v_2,
\end{cases}$$
if $N_{s,\ga}=1$, then $\ogs=1$, and for any $i \ge 2$ we have
\begin{eqnarray}
\left|\sum_{\substack{s \in \ov\\
N_{s,\ga}=i}} \ogs\right| & \le &  \sum_{\substack{s \in \ov\\
N_{s,\ga}=i}} \left|\prod_{i} \langle s\circ \ga(i), s \circ \ga(i+1) \rangle \right| \nonumber \\
& \le & \sum_{\substack{s \in \ov\\
N_{s,\ga}=i}} \left(\frac{1}{\sqrt{n}}\right)^{C_{s,\ga}} \le  \left(\frac{1}{\sqrt{n}}\right)^{C_{s,\ga}} (mn)^i \, \l_{\ga}^i. \label{4:wgs}
\end{eqnarray}
Next we provide an estimate of $C_{s,\ga}$. 
\begin{lemma} \label{4:lemn}
If $N_{s,\ga} \ge 2$, then
\begin{eqnarray} \label{4:csg1}
C_{s,\ga} \ge \max \left\{N_{s,\ga}, \,\,3N_{s,\ga} -v_{\ga}\right\}.
\end{eqnarray}
\end{lemma}
\begin{proof}
For fixed $\ga$ and $s$, we define an undirected graph $G=(V,E)$ as follows: the vertex set is $V=V_{\ga}$; as for the edge set $E$, for any $z,z' \in V$, the edge $\overline{zz'} \in E$ if and only if $s(z)=s(z')$ and there exists some index $i$ such that $\left\{\ga(i), \ga(i+1)\right\}=\{z,z'\}$.

If $G'=(V',E')$ is a connected component of $G$, then $s(z)=s(z')$ for any $z,z' \in V'$. Hence $N_{s,\ga}$ is the number of connected components of $G$. Now decompose $G$ into connected components
\[G=\bigcup_{t=1}^{N_{s,\ga}} G_t \]
where $G_t=(V_t,E_t)$.

For each $G_t$, noting that $V=\{\ga(i): 0 \le i \le \l_{\ga}-1\}$ and $\ga(\l_{\ga})=\ga(0)$,  there exists an index $t' \ne t$ with elements $z \in V_t$, $z' \in V_{t'}$ such that $\overline{zz'} \notin E$. This corresponds to an $i$ such that $s \circ \ga(i) \ne s \circ \ga(i+1)$. This show immediately that $C_{s,\ga} \ge N_{s,\ga}$. Moreover, if $\#V_t=1$ for some $t$, let $V_t=\{z\}$. Since $\ga$ is reduced, by (i) of Definition \ref{4:def1}, there exists at least two indices $i \ne j$ such that $\ga(i)=\ga(j)=z$. Since $\#V_t=1$, we also have $s \circ \ga(i) \ne s \circ \ga(i+1)$ and $s \circ \ga(j) \ne s \circ \ga(j+1)$. Now denote by $h$ the number of $t$'s such that $\#V_t=1$. From the above argument we have
\begin{eqnarray} \label{4:csg}
C_{s,\ga} &\ge & 2h+(N_{s,\ga}-h)=N_{s,\ga}+h.
\end{eqnarray}
As for $h$, since
\[ V=\bigcup_{t=1}^{N_{s,\ga}}V_t,\]
we have
\begin{eqnarray} \label{4:eqh}
v_{\ga} \ge h+2*(N_{s,\ga}-h) \quad \Longrightarrow \quad h \ge 2 N_{s,\ga}-v_{\ga}.\end{eqnarray}
Combining (\ref{4:csg}) and (\ref{4:eqh}) and noting that $h \ge 0$, we obtain the desired result (\ref{4:csg1}).
\end{proof}

Since $\Wg=\E\left(\ogs,\ov\right)$, we can write 
\begin{align}
\Wg&=\frac{1}{(mn)^{\vg}}\sum_{i=1}^{\vg}  \sum_{\substack{s \in \ov \\ N_{s,\ga}=i}} \ogs\nonumber\\
&=\frac{1}{(mn)^{\vg}}\left(\sum_{\substack{s \in \ov \\ N_{s,\ga}=1}} \ogs+
\sum_{2 \le i \le \frac{\vg}{2}} \sum_{\substack{s \in \ov \\ N_{s,\ga}=i}} \ogs +\sum_{\frac{\vg}{2}< i \le \vg}  \sum_{\substack{s \in \ov \\ N_{s,\ga}=i}} \ogs \right). \nonumber
\end{align}
Using (\ref{4:wgs}) and Lemma \ref{4:lemn}, we can obtain
\begin{align}
\Wg&\ll_{\l_{\ga}} \frac{1}{(mn)^{\vg}}\left(mn+\sum_{2 \le i \le \frac{\vg}{2}}(mn)^i n^{-\frac{i}{2}}+\sum_{ \frac{\vg}{2} <i \le v_{\ga}}(mn)^in^{-\frac{3i-\vg}{2}}\right). \nonumber
\end{align}
By the assumption that $m \geq \sqrt{n}$, we can easily conclude that
\begin{align}
\Wg &\ll_{\l_{\ga}} \frac{1}{(mn)^{\vg}}\left(mn+(m\sqrt{n})^{\frac{\vg}{2}}+n^\frac{\vg}{2}\left(\frac{m}{\sqrt{n}}\right)^{\vg}\right)\nonumber\\
&\ll_{\l_{\ga}} n^{1-\vg}\left(\frac{1}{m}+\frac{1}{n}\right). \label{Wg4}
\end{align}

\subsection{Reduction for $\gamma$}

If $\ga$ is not reduced, then either $\gamma(u)=\ga(u+1)$ for some $u$, or $\#I_a =1$ for some $a$.

\textbf{Case 1:} Suppose $\ga(u)=\ga(u+1)$ for some $u$. Since for any $i$, $s \circ \gamma(i) \in \B$ is a unit vector, $\left\langle s \circ \gamma(u), s \circ \gamma(u+1) \right\rangle=1$. Defining the closed path $\ga':[0\isep \lga-1] \to [1\isep p]$ by
\begin{eqnarray} \label{4:nga}
\ga'(i)=\begin{cases}
\ga(i) & 0 \leq i \leq u-1, \\
\ga(i+1) & u \leq i \leq \lga-2,
\end{cases}
\end{eqnarray}
we see that
\[\prod_{i=0}^{\lga-1} \left\langle s\circ\ga(i),s\circ\ga(i+1) \right\rangle = \prod_{i=0}^{\lga-2} \left\langle s\circ\ga'(i),s\circ\ga'(i+1) \right\rangle. \]
That is, $\ogs=\om_{\ga'}(s)$. Hence by reducing $\gamma$ to $\gamma'$, we have
\begin{eqnarray} \label{4:red1}
\Wg=W_{\ga'} \quad \mbox{ and } \quad \l_{\ga'}=\lga-1, \quad v_{\ga'}=\vg.
\end{eqnarray}

\textbf{Case 2:} Suppose $\# I_a=1$ for some $a$. Let $I_a=\{u\}$, this $u$ is the only index such that $\ga(u)=z_a$. We can write $\Wg$ as
\begin{align}
\Wg&=\frac{1}{(mn)^{\vg}} \sum_{s: V_{\ga} \setminus \{z_a\} \to \B } \prod_{i \notin \{u-1, u\}}\left \langle s \circ \ga(i), s \circ \ga(i+1) \right\rangle \nonumber \\
& \ \, \,\,\, \times \sum_{k=1}^m \sum_{s(z_a) \in B_k} \left\langle s \circ \ga(u-1), s(z_a) \right \rangle \cdot \left\langle s (z_a), s \circ \ga(u+1) \right \rangle . \nonumber
\end{align}
Since $B_k$ is an orthonormal basis of $\C^n$, it is easy to see that
\[\sum_{s(z_a) \in B_k} \left\langle s \circ \ga(u-1), s(z_a) \right \rangle \cdot \left\langle s (z_a), s \circ \ga(u+1) \right \rangle=\left\langle s \circ \ga(u-1), s \circ \ga(u+1) \right\rangle. \]
Thus we have
\begin{align}
\Wg&=\frac{m}{(mn)^{\vg}} \sum_{s: V_{\ga} \setminus \{z_a\} \to \B } \prod_{i \notin \{u-1, u\}}\left \langle s \circ \ga(i), s \circ \ga(i+1) \right\rangle \nonumber \\
& \ \, \,\,\, \times  \left\langle s \circ \ga(u-1), s \circ \ga(u+1) \right\rangle . \nonumber
\end{align}
Defining the closed path $\ga':[0\isep \lga-1] \to [1\isep p]$ as in (\ref{4:nga}) again, we still have $\ogs=\om_{\ga'}(s)$ and 
\begin{eqnarray} \label{4:red2}
\Wg=\frac{1}{n} \, W_{\ga'} \quad \mbox{ and } \quad \l_{\ga'}=\lga-1, \quad v_{\ga'}=\vg-1.\end{eqnarray}

\subsection{Estimate of $\Wg$}

The path $\ga$ in general may not be reduced. We can conduct reductions on $\ga$ repeatedly via either \textbf{Case 1} or via \textbf{Case 2}, as long as the resulting closed path is not reduced. Suppose altogether we have conducted \textbf{Case 1} reduction $u$ times and \textbf{Case 2} reduction $w$ times and finally we arrive at  a closed path $\widetilde{\ga}$ which is reduced. Then from (\ref{4:red1}) and (\ref{4:red2}) we have
\begin{equation}\label{nonre}
\l_{\widetilde{\ga}}=\lga-u-w, \quad v_{\widetilde{\ga}}=\vg-w\quad\mbox{and}\quad\Wg=\frac{1}{n^{w}} \, W_{\widetilde{\ga}}.
\end{equation}

If $\l_{\widetilde{\ga}}=v_{\widetilde{\ga}}=1$, then $W_{\widetilde{\ga}}=1$. Noting that $v_{\widetilde{\ga}}=1=v_{\ga}-w$, we have $\Wg=n^{1-v_{\ga}}$.

Denote by $\Ga_\l$ the set of all closed paths $\ga \in \plsp$ which can be reduced via \textbf{Case 1} or \textbf{Case 2} reductions to $\widetilde{\ga}$ with $\l_{\widetilde{\ga}}=v_{\widetilde{\ga}}=1$, which is a single point with a loop. Note that the same set $\Ga_{\l}$ has appeared in \cite[Section IV]{OQBT} and in the standard proof of the Marchenko-Pastur law for random matrices (see \cite{Ander, MET}), representing ``double trees''. Interested readers may refer to \cite{Ander, MET} for more detailed descriptions of the set $\Ga_\l$.

If $\ga \notin \Ga_\l$, then the resulting $\widetilde{\ga}$ is reduced with $v_{\widetilde{\ga}} \ge 2$. From (\ref{Wg4}) we have $W_{\widetilde{\ga}} \ll_{\l_{\widetilde{\ga}}} n^{1-v_{\widetilde{\ga}}} \left(\frac{1}{m}+\frac{1}{n}\right)$. Using (\ref{nonre}) we obtain
\[\Wg= \frac{1}{n^w} \, W_{\widetilde{\ga}} \ll_{\l_{\ga}} n^{1-v_{\ga}} \left(\frac{1}{m}+\frac{1}{n}\right). \]
Thus we have proved the desired result (\ref{2:wgamma}): 
\begin{lemma} \label{4:wglem}
If $\ga: [0 \isep \l] \to [1 \isep p]$ is a closed path, then 
\begin{eqnarray*}
\Wg=\begin{cases}
n^{1-\vg} &\mbox{ if } \ga \in \Ga_\l, \,  \\
O_{\l}\left(n^{1-\vg}\left(m^{-1}+n^{-1}\right)\right) & \mbox{ if } \ga \notin \Ga_\l\, .
\end{cases}\end{eqnarray*}
\end{lemma}


\section*{Acknowledgements} The research of M. Xiong was supported by RGC grant number 16303615 from Hong Kong.



\end{document}